\RequirePackage[l2tabu, orthodox]{nag}
\documentclass{scrartcl}
\usepackage{typearea}
\typearea{12}
\usepackage{amssymb, amsmath, amsthm}
\usepackage{enumitem}
\usepackage[all, warning]{onlyamsmath}
\usepackage[dvipdfmx]{graphicx}
\usepackage[ruled, lined]{algorithm2e}
\usepackage{tikz}\usetikzlibrary{arrows}
\usepackage{color}
\usepackage{authblk}
\usepackage{url}
\usepackage{ulem}

\newtheorem{theorem}{Theorem}[section]

\newtheorem{lemma}[theorem]{Lemma}

\makeatletter
\global\let\tikz@ensure@dollar@catcode=\relax
\makeatother

\title{Minimization of Akaike's Information Criterion  \\ in Linear Regression Analysis \\ via Mixed Integer Nonlinear Program}
\author[1]{Keiji Kimura\thanks{744 Motooka, Nishi-ku, Fukuoka 819-0395, Japan. k-kimura@math.kyushu-u.ac.jp}}
\author[2]{Hayato Waki\thanks{744 Motooka, Nishi-ku, Fukuoka 819-0395, Japan. waki@imi.kyushu-u.ac.jp}}
\affil[1]{Faculty of Mathematics, Kyushu University}
\affil[2]{Institute of Mathematics for Industry, Kyushu University}
\date{First version :  June 17, 2016, Revised : \today}

\begin{document}
\maketitle

\begin{abstract}
Akaike's information criterion (AIC) is a measure of the quality of a statistical model for a given set of data. We can determine the best statistical model for a particular data set by the minimization of the AIC. Since we need to evaluate exponentially many candidates of the model by the minimization of the AIC, the minimization is unreasonable. Instead, stepwise methods, which are local search algorithms, are commonly used to find a better statistical model though it may not be the best. 

We propose a branch and bound search algorithm for  a mixed integer nonlinear programming formulation of the AIC minimization  by Miyashiro and Takano (2015). 
More concretely,  we propose methods to find lower and upper bounds,  
 and branching rules for this minimization. We then combine them with SCIP, which is a mathematical optimization software and a branch-and-bound framework. We show that the proposed method can provide the best statistical model based on AIC for small-sized or medium-sized benchmark data sets in UCI Machine Learning Repository. Furthermore, we show that this method finds  good quality solutions for large-sized benchmark data sets. 

{\bfseries Keywords} : Mixed integer nonlinear program, branch-and-bound, SCIP and Akaike's information criterion
\end{abstract}
\section{Introduction}
Selecting the best statistical model from a number of candidate statistical models for a given set of data is one of the most important problems solved in statistical applications, {\itshape e.g.}  regression analysis. This is called {\itshape variable selection}. The purposes of variable selection are to provide the simplest statistical model for  a given data set and  to improve the prediction performance while keeping the goodness-of-fit for a given data set. See \cite{Guyon03} for more details on variable selection.  

In variable selection based on  {\itshape an information criterion}, all the candidates are evaluated by the information criterion and select a statistical model by using those  evaluations. 
Akaike's information criterion (AIC) is one of the information criteria and  proposed in \cite{Akaike74}. An AIC value is computed for each candidate, and the model whose AIC value is the smallest is selected as the best statistical model. Since we often need to handle too many candidates of statistical models in practical applications, the global minimization based on AIC is not practical. Instead of the global  minimization, stepwise methods, which are local search algorithms, are commonly used to find a statistical model which has  as small AIC as possible, but it may not be the smallest. 

The contribution of our study is to propose a branch and bound search algorithm for a mixed integer nonlinear programming (MINLP) formulation of the minimization of AIC in linear regression by Miyashiro and Takano \cite{Miyashiro15}. 
Miyashiro and Takano \cite{Miyashiro15} propose a mixed integer second-order cone programming (MISOCP) formulation from the MINLP formulation and solve  the resulting problems by CPLEX \cite{cplex}, while we propose  procedures to find lower and upper bounds of the MINLP problems and define branching rules for efficient computation. 
In addition, we provide an implementation to solve it efficiently via SCIP. SCIP is a mathematical optimization software and a branch-and-bound framework. SCIP has high  flexibility of user plugin and control on various parameters in the branch-and-bound framework for efficient computation. 
 We also propose an efficient computation for a set of data which has linear dependency. By applying our proposed method to benchmark data sets in \cite{ucimc},  we can obtain the best statistical models for some of them. Our implementation is available at \cite{maic}. 

We introduce some  related work. Miyashiro and Takano \cite{Miyashiro15} propose a MISOCP formulation  for variable selection based on some information criteria in linear regression. 
Bertsimas  and Shioda \cite{Bertsimas09} and Bertsimas, King and Mazumder  \cite{Bertsimas15} provide a mixed integer quadratic programming (MIQP) formulation for linear regression with a cardinality constraint. Their formulation is available to our problems by fixing the number of explanatory variables. We compare our proposed method with MIQP and MISOCP formulations, and observe that our proposed method outperforms MIQP and  MISOCP formulations. 

The organization of this manuscript is as follows: We give a brief introduction of linear regression based on AIC in Section 2. We introduce the MINLP formulation of the AIC minimization and  ways to find lower and upper bounds used in the branch-and-bound framework in Section 3. Section 4 introduces techniques for more efficient computation, {\itshape e.g.} branching rules and  treatment on data which has linear dependency.  We present numerical results in Section \ref{sec:numerical}. In particular, we show the numerical comparison with MISOCP and MIQP formulations. In addition, we present numerical performances of branching rules proposed  in subsection \ref{subsec:branch}.  We discuss future work of our proposed method in Section \ref{sec:futurework}. This manuscript is a full paper version of \cite{Kimura16}. 

\section{Preliminary on Akaike's information criterion in linear regression}\label{sec:AIC}
We explain how to select the best statistical model via AIC in linear regression analysis. Linear regression is a fundamental statistical tool which determines coefficients $\beta_0, \ldots, \beta_p\in\mathbb{R}$ for the following equation  from a given set of data:
\begin{equation}\label{Reg}
y = \beta_0 + \sum_{j=1}^p \beta_j x_{j}.  
\end{equation}
Here $x_1, \ldots, x_p$ and $y$  are called {\itshape the explanatory variables} and  {\itshape the response variable} respectively.  
In fact, we adopt coefficients $\beta_0, \ldots, \beta_p$ which minimize $\sum_{i=1}^n\epsilon_i^2$ for a given set of  data $(x_{i1}, \ldots, x_{ip}, y_i)\in\mathbb{R}^p\times\mathbb{R} \ (i=1, \ldots, n)$, where $\epsilon_i$ is the $i$th residual and defined by $\epsilon_i = y_i - \beta_0 -\sum_{j=1}^p \beta_jx_{ij}$.

Variable selection in linear regression is the problem to select the best subset of explanatory variables based on a given criterion. In statistical applications, a preferred model keeps the goodness-of-fit for a given data set, and contains as a few unnecessary explanatory variable as possible.  In fact, unnecessary explanatory variables may add the noise to the prediction based on the statistical model. As a result, the prediction performance of the model may get  worse. In addition, we need to observe and/or monitor more data for unnecessary explanatory variables, and thus  will spend more cost due to the unnecessary explanatory variables.

Akaike's information criterion (AIC) is one of  criteria for variable selection and  proposed in \cite{Akaike74}. 
AIC is used as a measure to select the preferred statistical model in all candidates. The statistical model whose AIC value is the smallest is expected as the preferred statistical mode. In  linear regression analysis, this selection corresponds to the selection of a subset of the set of explanatory variables in (\ref{Reg})  via AIC. More precisely, for a set $S\subseteq \{1, \ldots, p\}$ of candidates of explanatory variables in the statistical model (\ref{Reg}), AIC is defined in \cite{Akaike74} as follows: 
\begin{equation}\label{AIC}
\mbox{AIC}(S) = -2\max_{\beta, \sigma^2}\{\ell(\beta, \sigma^2) : \beta_j = 0 \ (j\in \{1, \ldots, p\}\setminus S)\} + 2(\#(S) + 2)
\end{equation}
where $\beta = (\beta_0, \ldots, \beta_p)\in\mathbb{R}^{p+1}$, $\#(S)$ stands for the number of elements in the set $S$ and $\ell(\beta, \sigma^2)$ is the log-likelihood function. 
Computing AIC values for all subsets $S$ of the explanatory variables in (\ref{Reg}), we can obtain the best AIC-based subset. However, since the number of subsets is $2^p$, the computation of all subsets is not practical.

Under assumption that all the residual $\epsilon_i$ are independent and normally distributed with the zero mean and variance $\sigma^2$, the log-likelihood function can be formulated as 
\[
\ell(\beta, \sigma^2) = -\frac{n}{2}\log(2\pi\sigma^2) -\frac{1}{2\sigma^2}\sum_{i=1}^n\epsilon_i^2.
\] 
We focus on the first term in (\ref{AIC}) to simplify (\ref{AIC}). Let $S$ be a set  of candidates of explanatory variables in (\ref{Reg}). By substituting $\beta_j = 0 \ (j\in \{1, \ldots, p\}\setminus S)$ to the objective function, the first term can be regarded  as  the unconstrained minimization. Thus minimum solutions satisfy the following equation
\[
\frac{d \ell}{d (\sigma^2)} = -\frac{n}{2\sigma^2} +\frac{1}{2(\sigma^2)^2}\sum_{i=1}^n\epsilon_i^2=0. 
\]
From this equation, we obtain $\sigma^2 = \frac{1}{n}\sum_{i=1}^n\epsilon_i^2$. 
Substituting this equation to (\ref{AIC}), we simplify (\ref{AIC}) as follows:
\begin{align}\label{AIC2} 
\mbox{AIC}(S) &= \min_{\beta_j}\left\{
n\log\left(\sum_{i=1}^n 
\epsilon_i^2\right): \beta_j = 0 \ (j\in\{1, \ldots, p\}\setminus S)
\right\}\\
&+ 2(\#(S) + 2)+ n\left(\log(2\pi/n) + 1\right). \nonumber 
\end{align}
We use (\ref{AIC2}) to provide our MINLP formulation of the minimization of  AIC in the next section. 

The following lemma ensures that the minimization in the first term of (\ref{AIC2}) has an optimal solution with a finite value. 
\begin{lemma}\label{existos}
For any  subset $S\subseteq\{1, \ldots, p\}$, the minimization in the first term of (\ref{AIC2}) has an optimal solution with a finite value. 
\end{lemma}
\begin{proof}
Since the logarithm function has the monotonicity, the optimal solution of the minimization in the first term of (\ref{AIC2}) is also optimal for the following unconstrained quadratic problem:
\begin{align}
\label{QP}
 &\min_{\beta_j}\left\{\sum_{i=1}^n \left(y_i - \beta_0 -\sum_{j\in S}\beta_j x_{ij}\right)^2: \beta_j \in \mathbb{R} \ (j\in\{0\}\cup S)
\right\}. 
\end{align}
Since the objective function of (\ref{QP}) is bounded below, it follows from \cite[Section 9.1.1]{Boyd04} that (\ref{QP}) has an optimal solution. 
\end{proof}

\section{MINLP formulation for the minimization of AIC}\label{sec:MINLP}
We provide the minimization of $\mbox{AIC}(S)$ over $S\subseteq\{1, \ldots, p\}$ by the following MINLP formulation:
\begin{equation}\label{MINLP}
\displaystyle\min_{\beta_j, z_j, \epsilon_i, k} \left\{n\displaystyle\log\left(\sum_{i=1}^n \epsilon_i^2\right) + 2k : \begin{array}{l}
\epsilon_i = y_i-\beta_0 -\displaystyle\sum_{j=1}^p\beta_j x_{ij} \ (i=1, \ldots, n),  \\
 \displaystyle\sum_{j=1}^p z_j = k, \beta_0, \beta_j\in\mathbb{R} \  (j=1, \ldots, p), \\
 z_j\in\{0, 1\}, z_j=0\Rightarrow \beta_j=0  \  (j=1, \ldots, p)
\end{array}
\right\}
\end{equation}
Here the last constraints represent the logical relationships, {\itshape i.e.}  $\beta_j$ has to be zero if $z_j=0$. This formulation is provided in \cite[eq. (22) -- (25)]{Miyashiro15}. 



Next we provide a procedure to find a lower bound of the subproblem of  (\ref{MINLP}) at each node in the branch-and-bound tree. Some variables $z_j$ in (\ref{MINLP}) are fixed to zero or one at each node of the tree. We define the sets $Z_0$, $Z_1$ and $Z$ for a given node as follows:
\begin{align*}
Z_1 &= \{j\in \{1, \ldots, p\} : z_j \mbox{ is fixed to } 1\}, Z_0 = \{j\in \{1, \ldots, p\} : z_j \mbox{ is fixed to } 0\}, \\
Z &= \{j\in \{1, \ldots, p\} : z_j \mbox{ is not  fixed}\}. 
\end{align*}
We remark that $Z_1\cup Z_0\cup Z = \{1, \ldots, p\}$ and that each set is disjoint with one another. In other words, we can uniquely specify a node in the branch-and-bound search tree by $Z_1, Z_0$ and $Z$. We denote the node by $V(Z_1, Z_0, Z)$. Then the subproblem at the node $V(Z_1, Z_0, Z)$ is formulated as follows:
\begin{equation}\label{subMINLP}
 \left\{
\begin{array}{cl}
\displaystyle\min_{\beta_j, z_j} & n\displaystyle\log\left(\sum_{i=1}^n 
\left(y_i-\beta_0 -\sum_{j=1}^p\beta_j x_{ij}\right)^2\right) + 2\displaystyle\sum_{j=1}^p z_j \\
\mbox{subject to} & z_j = 1 \ (j\in Z_1), z_j = 0 \ (j\in Z_0), z_j\in\{0, 1\} \ (j\in Z), \\
&  \beta_0, \beta_j\in\mathbb{R} \  (j=1, \ldots, p), \beta_j=0 \ (j\in Z_0) \  z_j=0\Rightarrow \beta_j=0  \ (j\in Z)
\end{array}
\right.  
\end{equation}
By relaxing the integrality of variables $z_j$ in (\ref{subMINLP}), we obtain the following relaxation problem:
\begin{equation}\label{relax0}
 \left\{
\begin{array}{cl}
\displaystyle\min_{\beta_j, z_j} & n\displaystyle\log\left(\sum_{i=1}^n 
\left(y_i-\beta_0 -\sum_{j=1}^p\beta_j x_{ij}\right)^2\right) + 2\displaystyle\sum_{j=1}^p z_j \\
\mbox{subject to} & z_j = 1 \ (j\in Z_1), z_j = 0 \ (j\in Z_0), 0\le z_j\le 1 \ (j\in Z), \\
&  \beta_0, \beta_j\in\mathbb{R} \  (j=1, \ldots, p), \beta_j=0 \ (j\in Z_0) \  z_j=0\Rightarrow \beta_j=0  \ (j\in Z)
\end{array}
\right.  
\end{equation}
Moreover we consider the following problem by eliminating all the logical relationships and all the  $z_j$: 
\begin{equation}\label{relax}
\displaystyle\min_{\beta_j} \left\{ n\displaystyle\log\left(\sum_{i=1}^n 
\left(y_i-\beta_0 -\sum_{j=1}^p\beta_j x_{ij}\right)^2\right) + 2\#(Z_1)  :
\begin{array}{l}
\beta_0, \beta_j\in\mathbb{R} \ (j\in Z\cup Z_1), \\
\beta_j = 0 \ (j\in Z_0)  
\end{array}
\right\}. 
\end{equation}

It should be noted that the optimal value of (\ref{relax}) is the same as the optimal value of (\ref{relax0}). Hence we deal with(\ref{relax}) as the relaxation problem of (\ref{subMINLP}). In fact, for the optimal solution $\beta^*$ of (\ref{relax}), we construct a sequence $\{(\beta^N, z^N)\}_{N=1}^{\infty}$ as follows: 
\[
\beta^N = \beta^* \mbox{ and } z_j^N = \left\{
\begin{array}{cl}
1 & \mbox{ if } j \in Z_1, \\
1/N & \mbox{ if } j\in Z \mbox{ and } \beta^N_j\neq 0,\\
0 & \mbox{ if } j\in Z \mbox{ and } \beta^N_j = 0, \\
0 & \mbox{ if } j\in Z_0,  
\end{array}
\right. \ (j=1, \ldots, p) 
\]
for all $N\ge 1$. Clearly, $(\beta^N, z^N)$ is feasible for (\ref{relax0}) for all $N\ge 1$. It is sufficient to prove that the objective value $\theta^N$ of (\ref{relax0}) at $(\beta^N, z^N)$ converges the optimal value $\theta^*$ of (\ref{relax}) as $N$ goes to $\infty$. Since we have 
\[
\theta^* \le \theta^N \le n\displaystyle\log\left(\sum_{i=1}^n 
\left(y_i-\beta_0^* -\sum_{j=1}^p\beta_j^* x_{ij}\right)^2\right) + 2\#(Z_1) + \frac{2}{N}\#(Z), 
\]
the right-hand side converges to $\theta^*$ as $N$ goes to $\infty$. This implies that the optimal value of (\ref{relax}) is the same as the optimal value of (\ref{relax0}).

Although the objective function of (\ref{relax}) contains the logarithm function, we can freely remove the constant $2\#(Z_1)$ and the logarithm by the monotonicity of the logarithm function in (\ref{relax}), and thus obtain the following problem from (\ref{relax}):
\begin{equation}\label{relax2}
\displaystyle\min_{\beta_j} \left\{\sum_{i=1}^n\left(y_i-\beta_0 -\displaystyle\sum_{j=1}^p\beta_j x_{ij}\right)^2 : \begin{array}{l}
\beta_0, \beta_j\in\mathbb{R} \ (j\in Z\cup Z_1)\\
\beta_j = 0 \ (j\in Z_0)
\end{array}
\right\}. 
\end{equation}
Since (\ref{relax2}) is the unconstrained minimization of a quadratic function, we can obtain an optimal solution of (\ref{relax2}) by solving a linear system. In our implementation, we call {\normalfont \ttfamily dposv}, which is a built-in function of  LAPACK \cite{lapack}  for solving the linear system. 
We denote the optimal value of (\ref{relax2}) by $\xi^*$. The optimal value of (\ref{relax}) is $n\log(\xi^*) + 2\#(Z_1)$, which is used as a lower bound of the optimal value of (\ref{subMINLP}). 

We provide a procedure that  constructs a feasible solution of (\ref{MINLP}) and computes an upper bound of the optimal value of (\ref{MINLP}). For this we use an optimal solution $\tilde{\beta}\in\mathbb{R}^{p+1}$ obtained after solving (\ref{relax2}).  We define 
\[
\tilde{z}_j = \left\{
\begin{array}{cl}
1 & (\mbox{if } j\in \tilde{Z}\cup Z_1), \\
0 & (\mbox{otherwise})
\end{array}
\right. \ (j=1, \ldots, p), \tilde{\epsilon}_i=y_i-\tilde{\beta}_0-\sum_{j=1}^p\tilde{\beta}_jx_{ij} \ (i=1, \ldots, n) \mbox{ and } \tilde{k}=\sum_{j=1}^p\tilde{z}_j, 
\]
where $\tilde{Z}=\{j\in Z : \tilde{\beta}_j\neq 0\}$. 
It is easy to see that $(\tilde{\beta}_j, \tilde{z}_j, \tilde{\epsilon}_i, \tilde{k})$ is feasible for (\ref{MINLP}) and the objective value is $n\log(\xi^*) + 2\#(\tilde{Z}\cup Z_1)$. If the objective value is smaller than the current best upper bound, then we update the current best upper bound. 


Finally, we give another understanding for our proposed formulation and propose an efficient computation based on this understanding. 
\begin{itemize}
\item Since we can regard (\ref{relax2}) as linear regression whose explanatory variables are in $Z_1\cup Z$, the computation of the lower bound from (\ref{relax2})  corresponds to the computation of the value  $\mbox{AIC}(Z_1\cup Z) - 2\#(Z)$, while the upper bound corresponds to  the AIC value of the statistical model whose explanatory variables are in $Z_1\cup Z$, {\itshape i.e.} $\mbox{AIC}(Z_1\cup Z)$. Therefore, our proposed method computes the AIC value of the the statistical model  with $Z_1\cup Z$ at each node $V(Z_1, Z_0, Z)$, up to constant term $4 + n(\log(2n\pi) + 1)$ of (\ref{AIC2}). In summary, we consider that our proposed method branches and prunes the branch-and-bound search tree efficiently by using this understanding. 

\item The statistical package {\ttfamily leaps} \cite{leaps} in {\ttfamily R} \cite{R} adopts the branch-and-bound scheme in a similar manner.  A QR decomposition is exploited at each node in the branch-and-bound search tree. In particular, {\ttfamily leaps} 
solve a linear system effectively by using the QR decomposition obtained at its parent node. 

{\ttfamily leaps} finds the best statistical model much faster than our proposed method for data sets whose $p$ is less than or equal to 32 and  which do not have linear dependency introduced in subsection \ref{subsec:lineardep}.  If the data set has linear dependency,  {\ttfamily leaps} does not work effectively, while our proposed method works more efficiently by using the linear dependency in data sets.  This technique will be discussed in Section \ref{subsec:lineardep}.

\item We  provide an efficient computation of lower and upper bounds based on this understanding.  We assume that we obtain the lower and upper bounds at  a node $V(Z_1, Z_0, Z)$. Then we do not need to solve (\ref{relax2}) at its child node $V(Z_1\cup\{j\}, Z_0, Z\setminus\{j\})$, where $j\in Z$. This node is generated by branching $z_j=1$ at the node $V(Z_1, Z_0, Z)$. In fact, since we have $(Z_1\cup\{j\})\cup (Z\setminus\{j\}) = Z_1\cup Z$, the relaxation problem (\ref{relax2}) at the child node $V(Z_1\cup\{j\}, Z_0, Z\setminus\{j\})$ is equivalent to one at the node $V(Z_1, Z_0, Z)$. Thus the upper  bound at the child node  is the same as one at the node $V(Z_1, Z_0, Z)$,  and the lower bound is the lower bound computed at the node $V(Z_1, Z_0, Z)$ plus two because of $2\#(Z_1\cup\{j\}) = 2\#(Z_1) + 2$. 

\end{itemize}

\section{Some techniques to improve the numerical performance}\label{sec:technique}
We describe some techniques to improve numerical performance to solve (\ref{MINLP}). 
\subsection{SCIP}
In order to implement our proposed method, we use SCIP \cite{Achterberg09,scip,Vigerske16}, which is a mathematical optimization software and a branch-and-bound framework. In fact, it has high user plug-in flexibility which helps to solve  (\ref{MINLP}) efficiently. We implement a procedure,  which is called {\itshape relaxator} or {\itshape relaxation handler}, to obtain lower bounds as in Section \ref{sec:MINLP}. In addition, we also implement procedures to compute upper bounds via a method based on stepwise methods discussed in subsection \ref{subsec:stepwise} and to define branching rules described in subection \ref{subsec:branch}. 

\subsection{Handling the linear dependency in data}\label{subsec:lineardep}

We illustrate that we can efficiently compute the optimal value of (\ref{MINLP}) by using the linear dependency in data.  Although linearly independent data is often the assumption in standard statistical textbooks, practical data has often linear dependency, {\itshape e.g.} {\normalfont \ttfamily servo} and {\normalfont \ttfamily auto-mpg} in
UCI Machine Learning Repository \cite{ucimc}. 

For a set of given data $(x_{i1}, \ldots, x_{ip}, y_i)\in\mathbb{R}^p\times\mathbb{R} \ (i=1, \ldots, n)$, we denote
\[
x^0 =\left(
\begin{array}{l}
1 \\
\vdots\\
1
\end{array}
\right), x^j =\left(
\begin{array}{l}
x_{1j} \\
\vdots\\
x_{nj}
\end{array}
\right)   \ (j=1, \ldots, p). 
\]
We say that data has linear dependent variables  if the vectors $x^0, x^1, \ldots, x^p\in\mathbb{R}^n$ are linearly dependent. 

From the definition of the linear dependency in data,  we can reduce the computational cost for solving (\ref{relax2}) when the data has linearly dependency. At a node $V(Z_1, Z_0, Z)$, if there exists a subset $S\subseteq Z_1\cup Z$ such that the vectors $\{x^k : k\in S\cup\{0\}\}$, we can fix one of variables $z_j$ in $j\in S\cap Z$ to zero. In fact, since we have  $\sum_{j\in S\cup\{0\}}\alpha_j x^j = 0$ for some $(\alpha_j)_{j\in S\cup\{0\}}\neq 0$, we can removes one variable $z_j$ by substituting this equation to (\ref{relax2}). This implies that the number of variables in (\ref{relax2}) decrease, and thus we solve the linear equation with a fewer variables. 

Moreover we can prune some nodes efficiently by using the linear dependency. The following lemma ensures that we do not need to branch  $z_q=1$ for some $q\in Z$ if the data has the linear dependency. Thus we need to handle only $z_q = 0$ in this case. 

\begin{lemma}\label{lineardepend}
Assume that in (\ref{subMINLP}), there exists $q\in Z$ such that {the vector} $x^q$ and {vectors $\{x^j : j\in Z_1\cup\{0\}\}$} are linearly dependent.  Then an optimal solution of (\ref{subMINLP}) satisfies $z_q = 0$. 
\end{lemma}
\begin{proof}
Let $(\tilde{\beta}_j, \tilde{z}_j)$ be an optimal solution of (\ref{subMINLP}), and $\theta^*$ be the optimal value of (\ref{subMINLP}). Suppose that $\tilde{z}_q = 1$.  It follows from the assumption that there exists  $\alpha_j\in\mathbb{R} \ (j\in Z_1\cup\{0\})$ such that $(\alpha_j)_{j\in Z_1\cup\{0\}}\neq 0$ and 
\[
x^q = \sum_{j\in Z_1\cup\{0\}} \alpha_j x^j. 
\]
Then the following solution $(\hat{\beta}_j, \hat{z}_j)$ is also feasible for  (\ref{subMINLP}):
\[
\hat{\beta}_j =\left\{
\begin{array}{cl}
\tilde{\beta}_j + \tilde{\beta}_q\alpha_j & (\mbox{if }j\in (Z\setminus\{q\})\cup Z_1\cup\{0\}), \\
0 & (\mbox{otherwise})
\end{array}
\right.  
\mbox{ and }
\hat{z}_j = \left\{
\begin{array}{cl}
1 & (\mbox{if }j\neq q \mbox{ and }\tilde{z}_j = 1), \\
0 & (\mbox{otherwise})
\end{array}
\right.
\]
The objective value of (\ref{subMINLP}) at $(\hat{\beta}_j, \hat{z}_j)$ is $\theta^*-2$, which  contradicts  the optimal value $\theta^*$. 
\end{proof}

A given set of data which has linear dependency satisfies the assumption of Lemma \ref{lineardepend}. In fact,  there exists a subset $S\subseteq\{1, \ldots, p\}$ such that the vectors $\{x^k : k\in S\cup\{0\}\}$ are linearly dependent. Hence  Lemma \ref{lineardepend} ensures that we do not need to generate a child node by branching $z_q = 1$ at a node $V(Z_1, Z_0, Z)$ when $q\in S\cap Z$ and $S\setminus\{q\}\subseteq Z_1$.  

In addition, if there exists a subset $S\subseteq\{1, \ldots, p\}$ such that for every $j\in S$, the vectors $\{x^k : k\in \{0\}\cup(S\setminus\{j\})\}$ are linearly dependent, then we can prune some nodes before applying our proposed method to (\ref{MINLP}). In fact, it follows from the assumption on $S$ that for every $j\in S$ we do not need to branch $z_j=1$ at the node $V(Z_1, Z_0, Z)$. This implies that optimal solutions of (\ref{MINLP}) satisfy the following linear inequality:
\[
\sum_{j\in S} z_j \le \#(S) -1. 
\] 
By adding this inequality in (\ref{MINLP}), we do not generate any nodes in which $S\subseteq Z_1$ hold.  We execute a greedy algorithm in Algorithm \ref{greedy} to  find a collection $\mathcal{C}$ of such sets $S$.  

\normalem
\begin{algorithm}[H] \label{greedy}
\KwIn{Data $x^0$, $x^1, x^2, \ldots x^p\in\mathbb{R}^n$}
\KwOut{A collection $\mathcal{C}$ of sets of linearly dependent vectors}
$\mathcal{C}\longleftarrow \emptyset$, $S\longleftarrow \emptyset$\;
\For{$j\rightarrow 0$ \KwTo $p$}{
\eIf{the vectors $\{x^j : j\in \{0\}\cup S\cup\{j\}\}$ is linearly independent}{
$S\longleftarrow S\cup\{j\}$\;
}{
Solve the following linear equation:
\begin{align}\label{linequ}
\sum_{k\in S\cup\{0\} }\alpha_k x^k&= x^j. 
\end{align}
$S^{\prime}\longleftarrow \{k\in S : \alpha_k\neq 0\}$, $\mathcal{C}\longleftarrow\mathcal{C}\cup\{S^{\prime}\}$\;
}
}
\Return{$\mathcal{C}$}\;
\caption{A greedy algorithm to find a collection of sets of linearly dependent vectors}
\end{algorithm}
\ULforem
We remark that the linear equation (\ref{linequ}) has a unique solution because the matrix $(x^k)_{k\in S\cup\{0\}}$ is of full column rank. 

%
%
%
%

%
%

\subsection{Computation of upper bounds based on stepwise methods}\label{subsec:stepwise}

Although we mainly use the procedure described in Section \ref{sec:MINLP} to compute upper bounds, we also use 
 the stepwise methods with forward selection (SW$_+$) and backward elimination (SW$_-$). SW$_+$ starts with no explanatory variables and adds one explanatory variable at a time until the AIC value does not decrease. More precisely, for the current set $S$ of explanatory variables, we choose an explanatory variable whose the AIC value $\mbox{AIC}(S\cup\{j\})$ is minimized over $j \in\{1, \ldots, p\}\setminus S$. SW$_-$ is just the reverse of SW$_+$. It starts with all explanatory variables and remove one explanatory variable at a time until the AIC value does not decrease. Note that since these methods add or remove one explanatory variable at a time, they may miss the best statistical model. In this sense, we can say that they are local search algorithms for variable selection. 
 
We describe our heuristics  to computer an upper bound in more details in Algorithm \ref{stepwisealgo}. To this end, we define $S\subseteq Z_1\cup Z$ for subproblem (\ref{subMINLP}) and consider the following problem:
\begin{align}\label{stepwise}
& \left\{
\begin{array}{cl}
\displaystyle\min_{\beta_j, z_j} & n\displaystyle\log\left(\sum_{i=1}^n 
\left(y_i-\beta_0 -\sum_{j=1}^p\beta_j x_{ij}\right)^2\right) + 2\displaystyle\sum_{j=1}^p z_j \\
\mbox{subject to} & \beta_0, \beta_j\in\mathbb{R}, z_j = 1 \ (j\in S), \beta_j=0, z_j = 0 \ (j\in\{1, \ldots, p\}\setminus S)
\end{array}
\right.   
\end{align}
 We denote the optimal value and an optimal solution of (\ref{stepwise}) by $\bar{\theta}_S$ and $(\bar{\beta}_S, \bar{z}_S)$, respectively.  
\normalem
\begin{algorithm}[H] \label{stepwisealgo}
\KwIn{$Z_1, Z_0$ and $Z$}
\KwOut{A feasible solution $(\beta, z)$ of (\ref{subMINLP})}
\tcc{Stepwise method with forward selection}
$S\longleftarrow Z_1$, $v_f\longleftarrow \infty$\;
\While{$\bar{\theta}_S < v_f$}{
 $v_f\longleftarrow \bar{\theta}_S$, $(\beta_f, z_f)\longleftarrow (\bar{\beta}_S, \bar{z}_S)$\;
 Find $j\in Z\setminus S$ such that  $\bar{\theta}_{S\cup\{j\}}$ is minimized over all $j\in Z\setminus S$\;
 $S\longleftarrow S\cup\{j\}$\;
}
\tcc{Stepwise method with backward elimination}
$S\longleftarrow Z_1\cup Z$, $v_b\longleftarrow \infty$\;
\While{$\bar{\theta}_S < v_b$}{
 $v_b\longleftarrow \bar{\theta}_S$, $(\beta_b, z_b)\longleftarrow (\bar{\beta}_S, \bar{z}_S)$\;
 Find $j\in Z\cap S$ such that  $\bar{\theta}_{S\setminus\{j\}}$ is minimized over all $j\in Z\cap S$\;
 $S\longleftarrow S\setminus\{j\}$\;
}
\eIf{$v_f < v_b$}{
\Return{$(\beta_f, z_f)$}\;
}{
\Return{$(\beta_b, z_b)$}\;
}
\caption{Stepwise methods to compute an upper bound}
\end{algorithm}
\ULforem
 
 We remark that an optimal solution of (\ref{stepwise}) is feasible for the subproblem (\ref{subMINLP}) if $Z_1\subseteq S$. Since $S$ always contains $Z_1$ in Algorithm \ref{stepwisealgo}, the returned solution $(\beta, z)$ is feasible for (\ref{subMINLP}). In addition, we set $Z_1$ as the initial set of SW$_+$ instead of the empty set  because we execute Algorithm \ref{stepwisealgo} at the node $V(Z_1, Z_0, Z)$. Similarly, we set $Z_1\cup Z$ as the initial set of SW$_-$. These are different from the original stepwise methods.

 In statistical applications, instead of finding the global minimum of (\ref{MINLP}), stepwise methods, which are local search algorithms, are commonly used in practice.  In fact, they often find a better statistical model and work effectively in our implementation.   However since stepwise methods spend more computational costs than the procedure described in Section \ref{sec:MINLP},  we apply  Algorithm \ref{stepwisealgo} to only subproblem (\ref{subMINLP}) at  the node whose depth from the root node is less than or equal to 10 in our implementation.

\subsection{Most frequent branching and Strong branching}\label{subsec:branch}

We define two branching rules for variables $z_j$ to improve the performance of our implementation. The first one is called {\itshape most frequent branching} and uses all stored feasible solutions in the procedure to compute upper bounds. The second one is called {\itshape strong branching}. This is based on the strong branching rule in \cite[Section 5.4]{Achterberg07}. We propose a more efficient computation for the strong branching rule than \cite{Achterberg07}. We will show the numerical comparison with branching rules implemented in SCIP in subsection \ref{subsec:compbran}. We will observe from the numerical results that most frequent branching is effective for a set data which  has linear dependency, while strong branching is effective for a set data which does not have linear dependency. 

The most frequent branching is based on the tendency  that some explanatory variables adopted for the best statistical model are also used in  statistical models whose AIC value is close to the smallest AIC value. By branching  variables $z_j$ in (\ref{subMINLP}) which correspond to such explanatory variables, we can expect that (\ref{subMINLP}) at the node generated by $z_j = 0$  is pruned as early as possible. To find such explanatory variables, we use  feasible solution stored in  our procedure to compute upper bounds. 
We describe the most frequent branching rule at the current node in Algorithm \ref{mostfrequent}. 

\normalem
\begin{algorithm}[H] \label{mostfrequent}
\KwIn{A positive integer $N$, a set $Z$ of unfixed variables in the node and all feasible solutions of (\ref{MINLP}) found from the root node through the current node }
\KwOut{$J\in Z$}
Choose $N$ feasible solutions $(\beta^1, z^1),  \ldots, (\beta^N, z^N)$  out of all stored feasible solutions\;   
\tcc{Here $(\beta^i, z^i)$ is a feasible solution of (\ref{MINLP}) whose objective value is the $i$th smallest in all the stored solutions}
\For{$j\in Z$}{
Compute score value $s_j$  defined by
$s_j = \#(T_j)$, where  $T_j = \{\ell\in \{1, \ldots, N\} : z^{\ell}_j = 1\}$\;
}
\Return{$J\in Z$ with $s_{J}=\displaystyle\max_{j\in Z }\{s_j\}$}\;
\caption{Most frequent branching rule}
\end{algorithm}
\ULforem
We observe in our preliminary numerical experiment  that the obtained lower bound at the child node generated by $z_{J}=0$ tends to be relatively bigger and that the pruning process tends to work earlier in comparison to branching rules of SCIP. As a result,  our proposed method with the most frequent branching rule often visits a fewer nodes in the branch-and-bound tree. 

In the strong branching rule, we  compute lower bounds for all possible branching $z_k=1$ and $z_k=0$, and choose index $k\in Z$ so that the lower bound is maximized in all computed lower bounds. More precisely, for the subproblem (\ref{subMINLP}) at a node $V(Z_1, Z_0, Z)$ and  $k\in Z$, the relaxation problem of the subproblem branched by $z_k=1$ and $z_k=0$ can be formulated as (\ref{strong1}) and (\ref{strong0}) as follows, respectively.  
\begin{align}
\label{strong1}
& \left\{
\begin{array}{cl}
\displaystyle\min_{\beta_j} & n\displaystyle\log\left(\sum_{i=1}^n 
\left(y_i-\beta_0 -\sum_{j=1}^p\beta_j x_{ij}\right)^2\right) + 2\#(Z_1\cup\{k\}) \\
\mbox{subject to} & 
\beta_0, \beta_j\in\mathbb{R} \ (j\in (Z\setminus\{k\})\cup (Z_1\cup\{k\})), \beta_j = 0 \ (j\in Z_0) 
\end{array} 
\right. \\
\label{strong0}
& \left\{
\begin{array}{cl}
\displaystyle\min_{\beta_j} & n\displaystyle\log\left(\sum_{i=1}^n 
\left(y_i-\beta_0 -\sum_{j=1}^p\beta_j x_{ij}\right)^2\right) + 2\#(Z_1) \\
\mbox{subject to} & \beta_0, \beta_j\in\mathbb{R} \ (j\in (Z\setminus\{k\})\cup Z_1), \beta_j = 0 \ (j\in Z_0\cup\{k\})  
\end{array}
\right. 
\end{align}

Since we have $(Z\setminus\{k\})\cup (Z_1\cup\{k\})=Z\cup Z_1$, the optimal value of (\ref{strong1}) for all $k\in Z$ is $\theta^*+2$, where $\theta^*$ is the optimal value of (\ref{relax})  at a node $V(Z_1, Z_0, Z)$. Hence we select an index $k\in Z$ only from  all optimal values $\theta^*_k$ of (\ref{strong0}). We describe the strong branching rule at the current node in Algorithm \ref{strongbranch}. 

\normalem 
\begin{algorithm}[H] \label{strongbranch}
\KwIn{Subproblem (\ref{subMINLP}) in the node $V(Z_1, Z_0, Z)$}
\KwOut{$J\in Z$}
\For{$k\in Z$}{
Solve (\ref{strong0}) with $k$ and obtain optimal value $\theta^*_k$\;
}
\Return{$J\in Z$ with $\theta^*_J = \displaystyle\max_{k\in Z}\left\{\theta^*_k\right\}$}\;
\caption{Strong branching rule}
\end{algorithm}
\ULforem

\section{Numerical experiments}\label{sec:numerical}

We implement our approach and procedures discussed in Sections \ref{sec:MINLP} and  \ref{sec:technique}, and  apply our implementation\footnote{This  is available at \cite{maic}.}  to benchmark data sets in \cite{ucimc}.  We apply our implementation to standardized data sets, {\itshape i.e.}  the data is  transformed to have the zero mean and unit variance. Note that the standardized data has also linear dependency even if we apply the standardization to the original data which has linear dependency. The specification of the computer 
is CPU : 3.5 GHz Intel Core i7, Memory : 16GB and OS : OS X 10.9.5. In subsections \ref{subsec:minlp} and \ref{subsec:miqp}, we adopt the most frequent branching rule for data which has linear dependency, while we adopt the strong branching for data which does not have linear dependency. In subsection \ref{subsec:compbran}, we discuss the reason why we use the different branching rules.

\subsection{Comparison with stepwise methods and MISOCP approach}\label{subsec:minlp}
We compare our proposed method with stepwise methods (SW$_+$ and SW$_-$) and the MISOCP approach proposed in \cite{Miyashiro15} via CPLEX \cite{cplex}. This approach is also obtained from (\ref{MINLP}).  Although the objective function of (\ref{MINLP}) is non-convex, the difficulty due to the non-convexity is overcome by using  
  the identity $\exp(\log(x)) = x$ and the monotonicity of the exponential function $\exp(x)$. See \cite[Section 3.2]{Miyashiro15} for the detail. The resulting problem is formulated as  MISOCP and is tractable by CPLEX. 

Table \ref{tables} shows the summary of  numerical comparisons. The mark $\bullet$ in the first column indicates that the data has  linear dependency. The second, third, and sixth columns indicate the numbers of data, the explanatory variables in the  statistical model (\ref{Reg}),  and the ones in the models found by using each method. The fifth column indicates the obtained AIC values by each method.  The values with the bold font are the best among four values. The seventh column indicates the cpu time in seconds to compute the optimal value. ``$>$5000" means that the corresponded method  cannot find the optimal value within 5000 seconds.  The last column indicates the gap in the percent  as follows:
\[
\mbox{gap} =\displaystyle\frac{\mbox{upper bound} - \mbox{lower bound}}{\max\{1, |\mbox{upper bound}|\}}\times 100. 
\]
It should be noted that if the gap  is sufficiently close to zero, then the obtained value is optimal. MINLP, MISOCP, SW$_+$ and SW$_-$ indicate the results obtained by our proposed method, MISOCP approach and the stepwise method with forward selection and backward elimination, respectively. We observe the following from Table \ref{tables}. 
\begin{itemize}
\item MINLP computes the optimal value much 
faster than MISOCP. MINLP finds smaller AIC values than MISOCP even when MINLP cannot find them within 5000 seconds. 
\item The AIC value obtained by SW$_+$ or SW$_-$ is equal to one by MINLP, {\itshape i.e.} {\ttfamily  crime} and {\ttfamily forestfires}. In fact, as we mentioned in subsection \ref{subsec:stepwise}, we use stepwise methods in some nodes in our implementation. This implies that our procedure to compute an upper bound discussed in Section \ref{sec:MINLP} cannot find  better feasible solutions than ones by the stepwise methods. 
\end{itemize}
\begin{table}[htbp]
\centering
\begin{tabular}{llllrrrr}\hline
 Name &$n$ & $p$ &Methods &AIC &$k$ &time($\sec$) &gap(\%)\\ \hline
 housing & 506 & 13 & MINLP & \bfseries 776.21 & 11 & 0.04&0.00\\ 
 &&& MISOCP & \bfseries 776.21 & 11 & 7.96 & 0.00\\ 
 &&& SW$_+$ & \bfseries 776.21 & 11 & 0.35 &--\\
 &&& SW$_-$ & \bfseries 776.21 & 11 & 0.10&--\\ 
\hline

 $\bullet$servo &  167 & 19 & MINLP & \bfseries 258.35 & 9 & 0.79& 0.00\\
 &&& MISOCP & \bfseries 258.35 & 9 & 7.99 & 0.00\\
 &&& SW$_+$ & \bfseries 258.35 & 9 & 0.19&--\\
&&& SW$_-$ & 260.16 & 10 & 0.18&-- \\ 
\hline

 $\bullet$auto-mpg &  392 & 25 & MINLP & \bfseries 332.88 & 15 & 1.76& 0.00\\
  &&& MISOCP & \bfseries 332.88 & 15 & 303.83 & 0.00\\
 &&& SW$_+$ & 334.73 & 16 & 0.49&--\\
 &&& SW$_-$ & 337.96 & 18 & 0.32&--\\ 
\hline

 $\bullet$solarflareC & 1066 & 26 & MINLP & \bfseries 2816.29 & 9 & 10.49&0.00\\
 &&& MISOCP &\bfseries 2816.29 & 9 & 304.51 & 0.00\\
 &&& SW$_+$ & \bfseries 2816.29 & 9 & 0.45&-- \\
&&& SW$_-$ & 2821.61 & 12 & 1.08 &-- \\ 
\hline
 $\bullet$solarflareM & 1066 & 26 & MINLP & \bfseries 2926.90 & 7 & 3.99&0.00\\
 &&& MISOCP &\bfseries 2926.90 & 7 & 255.02& 0.00\\
 &&& SW$_+$ & \bfseries 2926.90 & 7 & 0.36&--\\
&&& SW$_-$ & 2930.91 & 9 & 1.16&--\\
 \hline
$\bullet$solarflareX & 1066 & 26 & MINLP & \bfseries 2882.80 & 3 & 0.92&0.00\\
 &&& MISOCP &\bfseries 2882.80 & 3 & 19.39& 0.00\\
 &&& SW$_+$ & \bfseries 2882.80 & 3 & 0.18&--\\
 &&& SW$_-$ & 2891.56 & 9 & 1.20&-- \\ 
\hline

 breastcancer &  194 & 32 & MINLP & \bfseries 508.40 & 10 & 90.21&0.00\\
 &&& MISOCP & 508.62 & 10 & $>$5000 & 3.72\\
 &&& SW$_+$ & 509.50 & 8 & 0.24&--\\
 &&& SW$_-$ & 509.96 & 14 & 0.60&--\\ 
\hline

 $\bullet$forestfires & 517 & 63 & MINLP & \bfseries 1429.64 & 12 & $>$5000 &  0.77\\
  &&& MISOCP & 1431.32 & 12 & $>$5000 & 6.44 \\
 &&& SW$_+$ & \bfseries 1429.64 & 12 & 0.94&--\\ 
 &&&SW$_-$ & 1447.36 & 21 & 7.43&--\\ 
\hline

 $\bullet$automobile & 159 & 65 & MINLP & \bfseries -61.28 & 32 & $>$5000 & 13.95\\
  &&& MISOCP & -55.83 & 34 & $>$5000 & 27.22 \\
 &&& SW$_+$ & -28.55 & 21 & 1.12&--\\ 
 &&& SW$_-$ & -47.61 & 40 & 2.64&--\\
 \hline

 crime & 1993 & 100 & MINLP & \bfseries 3410.25 & 50 & $>$5000 &  0.50\\
 &&& MISOCP & 3469.34 & 74 & $>$5000 & 8.51\\
 &&& SW$_+$ & 3430.19 & 37 & 17.03&--\\ 
 &&&SW$_-$ & \bfseries 3410.25 & 50 & 105.40&--\\ 
\hline
\end{tabular}
\caption{Summary of numerical results by MINLP, MISOCP, SW$_+$ and SW$_-$}
\label{tables}
\end{table}

\subsection{Comparison of branching rules}\label{subsec:compbran}
We compare the numerical performance of the most frequent branching and strong branching with branching rules implemented in SCIP. In Table \ref{tablesb}, Std, MFB and SB stand for numerical results by the branching rules in SCIP, the most frequent branching rule and the strong branching rule. The sixth column indicates the number of visited nodes by our proposed method with the applied branching rule. 
The values with the bold font are the best among three values. 
We observe from Table \ref{tablesb}:
\begin{itemize}
\item The most frequent branching rule works more effectively than other ones for sets of data which have linear dependency. In fact, the gap by the most frequent branching rule is the smallest and the computation time is the shortest. In addition, The number of the visited nodes by the most frequent branching is also smaller than  other branching rules. In contrast, the strong branch is more efficient than other branching rules for data which do not have linear dependency. 

\item For $p\le 32$, the gap obtained by the best branching rule is the smallest in three branching rules, though it  visits fewest nodes in the branch-and-bound tree. This means that the best branching computes tighter lower bounds than other branching rules. 
\item These are the reasons why we use different branching rules in Tables \ref{tables} and \ref{tablesqp}.
\end{itemize}

\begin{table}[htbp]
\centering
\begin{tabular}{llrrrrr}\hline
 Name &Methods &AIC &$k$ &time($\sec$) & \# of visited nodes &gap(\%)\\ \hline
 housing 
 &Std &\bfseries 776.21 &11 &0.05 &55 &0.00\\ 
 &MFB &\bfseries 776.21 &11 &0.05 &49&0.00\\ 
 &SB &\bfseries 776.21 &11 &\bfseries 0.04&\bfseries 27 &0.00\\ 

\hline

 $\bullet$servo
 &Std &\bfseries 258.35 &9 &1.17 &7577 &0.00\\ 
 &MFB &\bfseries 258.35 &9 &0.79 &4705 &0.00\\ 
 &SB &\bfseries 258.35 &9 &\bfseries 0.41 &\bfseries 2261 &0.00\\ 

\hline

 $\bullet$auto-mpg
 &Std &\bfseries 332.88 &15 &4.06 &18959 &0.00\\ 
 &MFB &\bfseries 332.88 &15 &\bfseries 1.76 &\bfseries 5723 &0.00\\ 
 &SB &\bfseries 332.88 &15 &2.68 &11586 &0.00\\ 

\hline

 $\bullet$solarflareC 
 &Std &\bfseries 2816.29 &9 &53.33 &166639 &0.00\\ 
 &MFB &\bfseries 2816.29 &9 &\bfseries 10.49 &\bfseries 32261 &0.00\\ 
 &SB &\bfseries 2816.29 &9 &23.13 &79015 &0.00\\ 

\hline
 $\bullet$solarflareM
 &Std &\bfseries 2926.90 &7 &40.03 &117889 &0.00\\ 
 &MFB &\bfseries 2926.90 &7 &\bfseries 3.99 &\bfseries 11903 &0.00\\ 
 &SB &\bfseries 2926.90 &7 &23.72 &81899 &0.00\\ 

 \hline
$\bullet$solarflareX
 &Std &\bfseries 2882.80 &3 &4.37 &9737 &0.00\\ 
 &MFB &\bfseries 2882.80 &3 &\bfseries 0.92 &\bfseries 1519 &0.00\\ 
 &SB &\bfseries 2882.80 &3 &3.40 &7453 &0.00\\ 

\hline

 breastcancer 
 &Std &\bfseries 508.40 &10 &505.70 &3851$\times 10^3$ &0.00\\ 
 &MFB &\bfseries 508.40 &10 &478.66 &3422$\times 10^3$ &0.00\\ 
 &SB &\bfseries 508.40 &10 &\bfseries 90.21 &\bfseries550$\times 10^3$ &0.00\\ 

\hline

 $\bullet$forestfires
 &Std &\bfseries 1429.64 &12 &$>$5000 &7480$\times 10^3$ &1.11\\ 
 &MFB &\bfseries 1429.64 &12 &$>$5000 &13179$\times 10^3$ &\bfseries 0.77\\ 
 &SB &\bfseries 1429.64 &12 &$>$5000 &9938$\times 10^3$ &0.95\\ 
\hline

 $\bullet$automobile
 &Std &-60.29 &32 &$>$5000 &32192$\times 10^3$ &\bfseries 12.30\\ 
 &MFB &-61.28 &32 &$>$5000 &29785$\times 10^3$ &13.95\\ 
 &SB &\bfseries -61.59 &33 &$>$5000 &15300$\times 10^3$ &16.43\\ 
 \hline

 crime
 &Std &\bfseries 3410.25 &50 &$>5000$ &10272$\times 10^3$ &0.78\\ 
 &MFB &\bfseries 3410.25 &50 &$>5000$ &9753$\times 10^3$ &0.52\\ 
 &SB &\bfseries 3410.25 &50 &$>5000$ &1904$\times 10^3$ &\bfseries 0.50\\ 

\hline
\end{tabular}
\caption{Summary of numerical results by branching rules in SCIP (Std), the most frequent branching (MFB) and strong branching (SB)}
\label{tablesb}
\end{table}

\subsection{Comparison with MIQP formulation}\label{subsec:miqp}

Bertsimas  and Shioda \cite{Bertsimas09} and Bertsimas et al \cite{Bertsimas15} provide a mixed integer quadratic programming (MIQP) formulation with a cardinality constraint  for linear regression. Their formulation is available to the minimization of AIC by fixing the number of explanatory numbers from 0 to $p$. In fact, the minimization can be equivalently reformulated as follows:
\begin{align}\label{AIC3}
& \displaystyle\min_{k=0, \ldots, p}\min_{S \subseteq\{1, \ldots, p\}}\left\{
\mbox{AIC}(S) : \#(S) = k
\right\}. 
\end{align}
Since each inner optimization problem in (\ref{AIC3}) can be formulated as a MIQP problem, we can obtain the best statistical model by solving all $(p+1)$ optimization problems. In this subsection, we introduce a MIQP formulation by Bertsimas  and Shioda \cite{Bertsimas09} and Bertsimas et al \cite{Bertsimas15} for the inner optimization problems in (\ref{AIC3}). In addition, we provide a more efficient algorithm than this naive algorithm and compare the algorithm with our proposed method. 

Each inner optimization problem in (\ref{AIC3}) can be reformulated as follows: 
\begin{align}\label{AIC3k}
&\displaystyle\min_{\beta_j} 
\left\{
n\displaystyle\log\left(\sum_{i=1}^n \epsilon_i^2\right) + 2k : 
\begin{array}{l}
 \epsilon_i = y_i-\beta_0 -\displaystyle\sum_{j=1}^p\beta_j x_{ij} \ (i=1, \ldots, n), \\
\displaystyle\sum_{j=1}^p z_j = k, \beta_0, \beta_j\in\mathbb{R}\  (j=1, \ldots, p), \\
z_j\in\{0, 1\}, z_j=0\Rightarrow \beta_j=0  \  (j=1, \ldots, p)
\end{array}
\right\}
\end{align}

For any fixed $k$, since  the logarithm function in (\ref{AIC3k}) has the monotonicity, we can find an optimal solution (\ref{AIC3k}) by solving the following quadratic programming problem:
\begin{align}\label{QP2}
&\displaystyle\min_{\beta_j} 
\left\{
\displaystyle\sum_{i=1}^n\left(y_i-\beta_0 -\displaystyle\sum_{j=1}^p\beta_j x_{ij}\right)^2 : 
\begin{array}{l}
\displaystyle\sum_{j=1}^p z_j = k, \beta_0\in\mathbb{R}, \\
z_j\in\{0, 1\} \  (j=1, \ldots, p), \\
z_j=0\Rightarrow \beta_j=0  \  (j=1, \ldots, p)
\end{array}
\right\}
\end{align}
(\ref{QP2}) is a MIQP formulation. We denote the optimal value of (\ref{QP2}) by $\eta^*_k$. If (\ref{QP2}) is infeasible, we set $\eta^*_k=+\infty$.  Then  the optimal value of inner problem (\ref{AIC3k}) with $k$ is $n\log(\eta^*_k) + 2k$.  Therefore we obtain the optimal value and solution of (\ref{AIC3}) by computing all optimal values of (\ref{AIC3k}) for $k=0, \ldots, p$. We describe the naive algorithm in Algorithm \ref{naive}. 

\normalem
\begin{algorithm}[H] \label{naive}
\KwIn{Minimization of AIC (\ref{MINLP})}
\KwOut{An optimal solution of (\ref{MINLP})}
\For{$k\rightarrow 0$ \KwTo $p$}{
 Find the optimal value $\eta^*_k$ and an optimal  solution $(\beta^*_k, z^*_k)$ of   (\ref{QP2}) with $k$\;
 }
Find an index $K$ with $\theta^*_{K} = \displaystyle\min_{k=0, \ldots, p}\{n\log(\eta^*_k) + 2k\}$\;
\Return{$(\beta^*_{K}, z^*_{K})$}\; 
\caption{Naive algorithm for (\ref{MINLP}) via MIQP}
\end{algorithm}
\ULforem

The following lemma ensures that we can find an upper bound of $k$ if we have a feasible solution of (\ref{AIC2}). 

\begin{lemma}\label{update}
Let $\hat{\theta}\in\mathbb{R}^{p+1}$ be the optimal value of the following optimization problem:
\begin{align}\label{QP3}
& \displaystyle\min_{\beta_j}\left\{
n\log\left(\sum_{i=1}^n \left(y_i - \beta_0 - \sum_{j=1}^p \beta_{j}x_{ij}\right)^2\right) : \beta_0, \beta_j\in\mathbb{R} \ (j=1, \ldots, p)
\right\}. 
\end{align}
In addition, $\bar{\theta}$ is the objective value of (\ref{MINLP}) at a feasible solution of (\ref{MINLP}). Then any optimal solution $(\beta^*, z^*)$ of (\ref{MINLP}) satisfies 
\[
\sum_{j=1}^p z_j^* \le \left\lfloor \frac{\bar{\theta}-\hat{\theta}}{2}\right\rfloor. 
\]
\end{lemma}
\begin{proof}
Let $\theta^*$ be the optimal value of (\ref{MINLP}) and $(\beta^*, z^*)$ be an optimal solution of (\ref{MINLP}). Then we have 
\[
\bar{\theta} \ge \theta^*= n\log\left(\sum_{i=1}^n \left(y_i - \beta_0^* - \sum_{j=1}^p \beta_{j}^*x_{ij}\right)^2\right) + 2\sum_{j=1}^p z_j^*\ge \hat{\theta} +  2\sum_{j=1}^p z_j^*, 
\]
and thus we have $\sum_{j=1}^p z_j^* \le (\bar{\theta}-\hat{\theta})/2$. Since $z_j^*$ is integer, we obtain the desired result. 
\end{proof}

We describe an algorithm based on Lemma \ref{update} in Algorithm \ref{fast}.

\normalem
\begin{algorithm}[H] \label{fast}
\KwIn{Minimization of AIC (\ref{MINLP})}
\KwOut{An optimal solution of (\ref{MINLP})}
Solve (\ref{QP3}) and obtain $\hat{\theta}$\;
$\bar{\theta}\longleftarrow+\infty$\;
\For{$k\rightarrow 0$ \KwTo $p$}{
 \If{$k> \left\lfloor \frac{\bar{\theta}-\hat{\theta}}{2}\right\rfloor$}{
 	Stop\;
 }
 Find the optimal value $\eta^*_k$ and solution $(\beta^*_k, z^*_k)$ of   (\ref{QP2}) with $k$\;
 \If{$\bar{\theta}\ge n\log(\eta^*_k) + 2k$}{
 $\bar{\theta}\longleftarrow n\log(\eta^*_k) + 2k$, $(\beta^*, z^*)\longleftarrow(\beta^*_k, z^*_k)$\;
 }
 }
\Return{$(\beta^*, z^*)$}\; 
\caption{Faster algorithm for (\ref{MINLP}) via MIQP}
\end{algorithm}
\ULforem

We give details on our numerical experiment. 
\begin{enumerate}[label=$\square$]
\item We solve (\ref{AIC3k}) by CPLEX. In particular, since the last constraints in (\ref{AIC3k}) represent the logical relationship between $z_j$ and $\beta_j$, we use {\itshape indicator} implemented in CPLEX to represent these constraints. 
\item 
We add linear inequalities in (\ref{AIC3k}) by applying Lemma \ref{lineardepend} to (\ref{AIC3k}) when a given set of data  has linear dependency. 
 See subsection \ref{subsec:lineardep} for the detail. 
\item We also solve optimization problems obtained by replacing the constraint $\sum_{j=1}^p z_j = k$  by $\sum_{j=1}^p z_j \le k$ 
in (\ref{AIC3k}). In Table \ref{tablesqp}, ``Fast$\le$" indicates that we solve those problems in Algorithm \ref{fast}, while ``Fast$=$" indicates that we solve (\ref{AIC3k}) in Algorithm \ref{fast}. By this replacement, we can use an optimal solution $(\beta^*_k, z_k^*)$ of the optimization problem with $k$ to compute an upper bound  of the optimization problem with $k+1$. 
\item We terminate if the corresponded method cannot find the best AIC value within 5000 seconds. In addition, the values with the bold font are the best among four values except for ``$>$5000" in the last column. 
\end{enumerate}

We provide numerical results on our proposed method, Algorithms \ref{naive} and \ref{fast} in Table \ref{tablesqp}. 
We observe the following from Table \ref{tablesqp}: 
\begin{itemize}
\item MINLP outperforms MIQP approaches. In particular, for larger $p$, MINLP obtains much better AIC values than MIQP approaches although all approaches cannot solve within 5000 seconds. 
\item The performance of Fast$\le$ is similar to Fast$=$,  though Fast$\le$ uses an initial upper bound. 
\end{itemize}

\begin{table}[htbp]
\centering
\begin{tabular}{llrrrr}\hline
 Name &Methods &AIC &$k$ &time($\sec$) \\ \hline
 housing
 &MINLP &\bfseries 776.21 &11 &\bfseries 0.04 \\ 
 &Naive &\bfseries 776.21 &11 &2.54 \\ 
 &Fast$=$ &\bfseries 776.21 &11 &2.15 \\ 
 &Fast$\le$ &\bfseries 776.21 &11 &2.43 \\ 

\hline

 $\bullet$servo
 &MINLP &\bfseries 258.35 &9 &\bfseries 0.79 \\ 
 &Naive &\bfseries 258.35 &9 &2.27 \\ 
 &Fast$=$ &\bfseries 258.35&9 &1.27 \\ 
 &Fast$\le$ &\bfseries 258.35&9 &1.29\\ 

\hline

 $\bullet$auto-mpg
 &MINLP &\bfseries 332.88 &15 &\bfseries 1.76 \\ 
 &Naive &\bfseries 332.88 &15 &22.22 \\ 
 &Fast$=$ &\bfseries 332.88 &15 &19.04 \\ 
 &Fast$\le$ &\bfseries 332.88 &15 &14.45 \\ 

\hline

 $\bullet$solarflareC
 &MINLP &\bfseries 2816.29 &9 &\bfseries 10.49 \\ 
 &Naive &\bfseries 2816.29 &9 &26.49 \\ 
 &Fast$=$ &\bfseries 2816.29 &9 &18.17 \\ 
 &Fast$\le$ &\bfseries 2816.29 &9 &15.03 \\ 

\hline
 $\bullet$solarflareM
 &MINLP &\bfseries 2926.90 &7 &\bfseries 3.99 \\ 
 &Naive &\bfseries 2926.90 &7 &25.27 \\ 
 &Fast$=$ &\bfseries 2926.90 &7 &8.15 \\ 
 &Fast$\le$ &\bfseries 2926.90 &7 &7.24 \\ 

 \hline
$\bullet$solarflareX
 &MINLP &\bfseries 2882.80 &3 &\bfseries 0.92 \\ 
 &Naive &\bfseries 2882.80 &3 &10.65 \\ 
 &Fast$=$ &\bfseries 2882.80 &3 &2.25 \\ 
 &Fast$\le$ &\bfseries 2882.80 &3 &2.40 \\ 

\hline

 breastcancer
 &MINLP &\bfseries 508.40 &10 &\bfseries 90.21 \\ 
 &Naive &\bfseries 508.40 &10 &420.44 \\ 
 &Fast$=$ &\bfseries 508.40 &10 &402.64 \\ 
 &Fast$\le$ &\bfseries 508.40 &10 &421.96 \\ 

\hline

 $\bullet$forestfires
 &MINLP &\bfseries 1429.64 &12 &$>$5000 \\ 
 &Naive &1435.07 &7 &$>$5000 \\ 
 &Fast$=$ &1435.07 &7 &$>$5000 \\ 
 &Fast$\le$ &1435.07 &7 &$>$5000 \\ 

\hline

 $\bullet$automobile
 &MINLP &\bfseries -61.28 &32 &$>$5000 \\ 
 &Naive &52.84 &8 &$>$5000 \\ 
 &Fast$=$ &52.84 &8 &$>$5000 \\ 
 &Fast$\le$ &52.84 &8 &$>$5000 \\ 

 \hline

 crime
 &MINLP &\bfseries 3410.25 &50 &$>$5000 \\ 
 &Naive &3646.35 &4 &$>$5000 \\ 
 &Fast$=$ &3646.35 &4 &$>$5000 \\ 
 &Fast$\le$ &3646.35 &4 &$>$5000 \\ 

\hline
\end{tabular}
\caption{Summary of numerical results by MINLP, Naive (Algorithm \ref{naive}), Fast$=$ and Fast$\le$ (Algorithm \ref{fast})}
\label{tablesqp}
\end{table}

\section{Conclusion}\label{sec:futurework}
We propose the MINLP formulation (\ref{MINLP}) of AIC minimization for linear regression, and implement it by using SCIP. We formulate an unconstrained optimization problem (\ref{relax}) as the relaxation problem of the subproblem (\ref{subMINLP}). As a result, a lower bound can be computed by solving a linear equation at each node. In addition,  an upper bound is the lower bound plus a constant, and a feasible solution is generated from a solution after solving the relaxation problem (\ref{relax}). 

We implement this procedure with SCIP because it has the high flexibility in the user plugin. In fact, we implement a relaxator to compute lower and upper bounds, and two branching rules to prune  subproblems efficiently. In addition, our implementation efficiently prunes and branches subproblems by using linear dependency in data set and two branching rules. As a result, we can obtain the best statistical models (\ref{Reg}) for $p\le 32$.  In addition, we observe that our implementation outperforms MISOCP approach \cite{Miyashiro15} and MIQP approaches \cite{Bertsimas09, Bertsimas15} in our numerical experiments.

Future work involves to apply our implementation to data sets with  larger $p$ and/or $n$.  A possible choice to accomplish this involves the use of parallel computation via ParaSCIP and FiberSCIP \cite{Shinano12}. Secondly, various non-AIC information criterion, {\itshape e.g.} BIC and Hannan-Quinn information criteria  are already proposed. By changing the objective function in (\ref{MINLP}), our proposed method can be applied  to these information criteria as well. 

\section*{Acknowledgements}
The second author was supported by JSPS KAKENHI Grant Numbers 26400203. 



\end{document}